\documentclass[12pt]{article}
 \usepackage{amssymb}
\usepackage{amsmath}
\usepackage{bm}
\usepackage[english]{babel}
\usepackage{color}

\usepackage[final]{epsfig}
\usepackage{amsfonts, epsfig, float}
\usepackage{graphicx}
\usepackage{amssymb,latexsym}
\usepackage{amsmath}
\usepackage{stmaryrd}
\usepackage{textcomp}
\usepackage{epsfig, float}

\usepackage{amsbsy}
\usepackage{amsthm}

\newtheorem{proposition}{Proposition}
\newtheorem{theorem}{Theorem}
\newtheorem{remark}{Remark}

\newtheorem{lemma}{Lemma}
\newtheorem{definition}{Definition}

\newtheorem{corollary}{Corollary}
\newcommand{\be}{\begin{equation}}
\newcommand{\ee}{\end{equation}}
\newcommand{\bea}{\begin{eqnarray}}
\newcommand{\eea}{\end{eqnarray}}

\newcommand{\btau}{\bm{\tau}}
\newcommand{\cop}{{\mathbb {COP}}(n)}

\newcommand{\cp}{{\mathbb {CP}}(n)}

\newcommand{\bt}{\boldsymbol{t}}
\newcommand{\by}{\boldsymbol{y}}
\newcommand{\bz}{\boldsymbol{z}}

\newcommand{\bc}{\boldsymbol{c}}
\newcommand{\bd}{\boldsymbol{d}}

\begin{document}

\title{Refined Estimates  on the Dimensions of Maximal Faces of Completely Positive Cones
\thanks{The work was  supported by  the   research program "Convergence-2030"(Republic Belarus), Task  1.07 and  by CIDMA (https://ror.org/05pm2mw36)
under the Portuguese Foundation for Science and Technology, FCT ( https://ror.org/00snfqn58), Grants
UID/04106/2025 (https://doi.org/10.54499/UID/04106/2025) and UID/PRR/04106/2025}}

\author{Kostyukova O.I.\thanks{Institute of Mathematics, National Academy of Sciences of Belarus, Surganov str. 11, 220072, Minsk,
 Belarus  ({\tt kostyukova@im.bas-net.by}).}  Tchemisova T.V.  \thanks{University of Aveiro, Campus Universitário de Santiago, 3800-198, Aveiro, Portugal  ({\tt tatiana@ua.pt})}}

\date{}

\maketitle

\begin{abstract}
The structure of maximal faces of the cone of completely positive matrices is still not well  understood
 in higher dimensions, mainly due to the lack of a general characterization of extreme exposed  rays of the
 copositive cone beyond small matrix orders. This paper contributes to the study of maximal faces of
the cone of completely positive matrices by establishing   sharper bounds  on their dimensions  than those currently available.
For every odd dimension $n$, we prove that the exact lower bound on the dimensions of   maximal faces of the cone of
 $n \times n$ completely positive matrices equals $n$. For even dimensions $n \geq 8$, we derive a new upper
estimate  for this lower bound  and  show that it lies between $n$ and $n+3$. These results substantially refine the previously known bounds.
		
\end{abstract}



\section{Introduction} \label{S.1}

The cones of copositive and completely positive matrices play a central role in  copositive programming (CoP)  that has attracted increasing attention due to its wide range of applications.
Many hard optimization problems, including several NP--hard combinatorial problems, admit exact reformulations
as copositive optimization problems; see, for example, \cite{AH2013,Bomze2012,Dur2010,Burer} and the references therein.
From a structural perspective,  CoP fits naturally into the general framework of convex conic optimization,
 where a convex objective function is minimized over the intersection of  an affine subspace with a convex cone.

 For convex conic optimization problems, a thorough understanding of the facial structure of
 the underlying cones is essential. In particular, facial properties play a fundamental role in the construction
and justification of regularization techniques, such as facial reduction algorithms, in the development of duality theory,
and in the derivation of optimality conditions; see,  {\it e.g.}   \cite{Dickinson, Borwein, Ramana,KT-Mathematics}.

 At present, only partial results are available concerning the
facial structure of the cones of copositive and completely positive matrices.
 Significant progress has been achieved for certain special classes of faces; see, e.g.,
 \cite{Dickinson,KT-Mathematics,Afonin,R-H,Zhang}. The present paper continues this line of research, with a
 particular emphasis on the study of faces of the cone of completely positive matrices.
The geometric properties of this cone, and in particular the structure and dimensions of its maximal faces, are of fundamental importance for copositive optimization, since the primal--dual formulation of a copositive optimization problem naturally involves both the copositive and completely positive cones.

 {Despite their importance, the geometric aspects of the completely positive cones remain poorly understood in higher dimensions.
 This difficulty is closely related to the lack of a general characterization of exposed extreme rays of the copositive cone ${\mathbb{COP}}(n)$
 for matrix orders $n \geq 7$. While for $n \leq 6$ all extreme rays of ${\mathbb{COP}}(n)$ are known \cite{Afonin},  allowing for a complete description of the maximal faces of its dual cone ${\mathbb{CP}}(n)$, such characterizations are not available in general (\cite{Zhang}). Consequently, only limited information is currently known about the maximal faces of ${\mathbb{CP}}(n)$ and their dimensions.}

In  \cite{Dickinson}, the problem of determining    the exact value of the   tight lower bound on
the dimension of a maximal face of the completely positive cone was identified as an open  \ question.
 Since then, this problem has  been investigated by several authors; see, in particular, \cite{Dickinson} and  \cite{Zhang}.
In  \cite{Zhang},  {the} authors    establish   connections between exposed rays of lower-order copositive cones and faces
of higher-order completely positive cones, leading to nontrivial estimates  of the
tight lower bounds on the dimensions of   maximal faces   of the cone $\cp$.

  In this paper, we advance this line of research by deriving sharper
	results on the dimensions of maximal faces of the cone ${\mathbb{CP}}(n)$.	Our  main results establish that the tight lower bound $low(n)$ on the dimensions of maximal faces of ${\mathbb{CP}}(n)$ equals $n$ for all odd $n\geq 5$, and for even $n \ge 6$, we prove that
$n \le low(n) \le n+3.$
The upper bound obtained significantly improves upon previously known results and emphasizes the clear distinction between the odd- and even-dimensional cases.

 The paper is organized as follows. Section~\ref{S.1} provides the introduction.
Section~\ref{S.2} presents the basic notation and definitions related to general convex cones,
as well as to the cones of copositive and completely positive matrices. We also review known representations
of the set of   zeros   of a given copositive matrix
and previously established estimates for the dimensions of maximal faces of completely positive cones.
Sections~\ref{S.3} and~\ref{S.4} contain the main results of the paper, namely refined estimates for
the tight lower bound on the dimensions of maximal faces of completely positive cones.
These sections treat separately the cases of matrix cones of odd and even dimensions. Finally, Section~\ref{S.5} offers concluding remarks.

\section{Basic notations and definitions}\label{S.2}

\subsection{Convex cones, their faces and some properties}

Let us first recall some basic definitions related to convex cones, which apply to general classes of cones, including cones of matrices.

Let $\mathfrak{X}$ be a  finite dimensional  vector  space, where the inner (scalar) product is denoted by $\langle \cdot, \cdot \rangle$  and $\textbf 0$ denotes the null element.

A  set $C  \subseteq  \mathfrak{X}$ is convex if for any $x, y \in C$ and any $\alpha \in [0,1]$, it
 holds $\alpha x+(1-\alpha ) y \in C.$

\vspace{2mm}

 Given a set $\mathcal{B}  \subseteq  \mathfrak{X}$, we denote by:

    \vspace{-2mm}

\begin{itemize}
    \item ${\rm conv}(\mathcal{B})$ its \emph{convex hull}, the smallest convex set containing $\mathcal{B}$;

    \vspace{-3mm}
		
    \item ${\rm cone}(\mathcal{B})$ (resp. ${\rm aff}(\mathcal{B})$) its \emph{conic} (resp.
    \emph{affine}) hull, {\it i.e.}, the set of all conic (resp. affine) combinations of points in $\mathcal{B}$;

    \vspace{-3mm}

    \item  $ {\rm span}(\mathcal{B})$ its {\it span} (the smallest subspace of $\mathfrak{X}$ containing $\mathcal{B}$), the set of all finite linear combinations of elements of $\mathcal{B}$:
$$ {\rm span}(\mathcal{B})= \left\{\sum_{i=1}^{k} \lambda_i b_i\;:\; k \in \mathbb{N},\;b_i \in \mathcal{B},\;\lambda_i \in \mathbb R \right\};$$

    \vspace{-4mm}

		\item ${\rm dim}(\mathcal{B})$ its  dimension, {\it i.e.}  the dimension of its affine hull.

\end{itemize}

A set $\mathcal{K}  \subseteq  \mathfrak{X}$ is  called  a {\it cone} if for any $x \in \mathcal{K}$ and any $\alpha> 0$, it holds $\alpha x\in \mathcal{K}.$ Given a cone $\mathcal{K} \subset \mathfrak{X}$, its dual cone $\mathcal{K}^*$ is  defined as
 \begin{equation*}\mathcal{K}^*=\{u\in \mathfrak{X}: \langle x,u\rangle \geq 0  \ \forall x\in  \mathcal{K}\}.\end{equation*}

 A cone $\mathcal{K} \subset \mathfrak{X} $ is said to be \emph{pointed} if $\mathcal{K} \cap (-\mathcal{K}) = \{\textbf{0}\},$
that is, $ \mathcal{K}$ contains no nontrivial linear subspace. It is called \emph{full-dimensional} if
$ \dim(\operatorname{span} \mathcal{K}) = \dim \mathfrak{X}$. A cone $ \mathcal{K}\subset \mathfrak{X} $ is said to be {\it proper} if it is closed, convex, pointed, and full-dimensional.

Let $ \mathcal{K} \subset \mathfrak{X}$ be a  convex  cone. We recall some standard definitions related to the {\it faces} of the cone that will be used in what follows.
 \begin{itemize}
 \item  A nonempty convex subset $\mathcal{F}$ of  $\mathcal{K}$ is   called a   {\it face}  of $\mathcal{K}$   if   {the}  inclusion $\alpha x + ( 1-  \alpha )y \in  \mathcal{F}, $ with$ \ x, y \in \mathcal{K}  $ and $ \ \alpha \in (0, 1),$ implies   $x, y \in  \mathcal{K}$.
Evidently, each face of  a closed convex cone  $\mathcal{K}$ is itself a cone.

  \vspace{-3mm}

 \item  {A face}  $\mathcal{F}$ of   $\mathcal{K}$  is called  {\it proper}    if  $\mathcal{F}\not=\emptyset$ and   $\mathcal{F} \neq \mathcal{K}$.

    \vspace{-3mm}

 \item A face $\mathcal{F}$  is a {\it maximal face} of  ${\mathcal K}$  if $\mathcal{F}\not={\mathcal K}$, and there does not
exist other face $\bar{\mathcal{F}}\not={\mathcal K}$ of ${\mathcal K} $  such that $\mathcal{F}\subset \bar{\mathcal{F}}.$

    \vspace{-3mm}

\item  A face $\mathcal{F}$ of  a closed convex cone  $\mathcal{K}$  is called {\it exposed} if it can be  presented as the intersection of $\mathcal{K}$ with a  supporting hyperplane, that is,  there exist $y \in \mathfrak{X}$ and $d \in \mathbb R$
such that for all $x \in \mathcal{K}$, it holds $\langle y, x\rangle   \geq    d$ and  $\langle y, x\rangle = d $ if  and only if $x \in \mathcal{F}. $

    \vspace{-3mm}

\item For a face $\mathcal{F}$  of the cone   $\mathcal K,$ denote by ${\rm dim}(\mathcal{F})$ its dimension.

    \vspace{-3mm}

\item A cone $ \mathcal{K} $ is called {\it facially exposed} if all its faces are exposed.
\end{itemize}

  It follows from the definitions above that   a face $\mathcal{F}$ of  a closed convex cone  $\mathcal{K} \subset  \mathfrak{X}$  is  exposed if and only if there exists an element $u\in \mathcal{K}^*$  such that $\mathcal{F}=\mathcal{K}\cap u^\bot.$

Let $ \mathcal{K} \subset \mathfrak{X} $ be a cone and let $v \in \mathcal{K} \setminus \{\textbf{0}\}$.
We define the ray {\it generated by} (or given by) $v $ as $
\mathbb{R}_+ v := \{ \alpha v \mid \alpha \ge 0\}$.
 A ray $\mathbb{R}_+ v$ is  an {\it extreme ray} of the cone
$\mathcal{K}$ if, for any $ x, y \in \mathcal{K}$,
$ x + y \in \mathbb{R}_+ v \quad $ implies$ \quad x, y \in \mathbb{R}_+ v.$
 We will say that a nonzero element $ v \in \mathcal{K}$  generates an
 exposed ray  of a proper cone $\mathcal{K}$ if   the ray $R^+v$ is  an exposed face  of $\mathcal{K}$.
In this case, the ray $R^+v$ is called an {\it exposed ray} of $\mathcal{K}$.

  An element $v\in \mathcal{K}$ is called extremal/ exposed if it generates an extreme/ exposed ray of the cone $ \mathcal{K}.$

\begin{theorem}\label{T1} [\cite{Dickinson}, Theorem 2.20] If ${\mathcal K}$ is a proper cone and $x\in R^+v$, where ${R^+v}$ is an exposed ray of ${\mathcal K}^*$
generated by some $v \in \mathcal{K}^*\setminus \{\textbf{0}\}$, then ${\mathcal F}:={\mathcal K}\cap x^\bot$
is a maximal face of ${\mathcal K}$.
\end{theorem}

\subsection{The cones of copositive and completely positive matrices}

Throughout the paper, we  use the following standard notation.  {For} $n\in \mathbb{N},$ {denote} $[n]:=\{1,2, ...,n\}$.
 Let $\mathbb{R}^n$ be the $n$-dimensional Euclidean space  equipped  with  the  standard basis  $\{\mathbf{e}_i : i \in [n]\}$.
  For $n>1, $  vectors in $\mathbb R^n$  are written in  lowercase boldface   letters.    Unless stated otherwise,
	a vector \( \textbf{u} \in \mathbb{R}^n \),   is understood to be  a column vector.

Given $\textbf{u}\in \mathbb R^n,$ the support of $\textbf{u}$ is the index set ${\rm supp}(\textbf{u}):=\{i\in {[n]}:
\ u_i\neq 0\}$, where $u_i$ is the $i$-th component of vector $\textbf{u}$.

Denote by $\mathbb{R}^n_+$ the nonnegative orthant and by $\mathbb{R}^n_{++}$ the strictly positive orthant in $\mathbb{R}^n$.
 For any vector $\textbf{u}\in \mathbb{R}^n_+$, we  write
 $\textbf{u}{\geq} \textbf{0}$ and for any $\textbf{u}\in \mathbb{R}^n_{++}$, we   write  $\textbf{u} > \textbf{0}$.

 Let $\mathbb{S}(n)$ denote the space of symmetric $n \times n$ matrices, and let $\mathbb{S}_+(n)$ denote the cone of symmetric positive semidefinite matrices.
Denote by $\mathbb{COP}(n)$ the cone of copositive $n \times n$ matrices
 $$\cop:=\{X\in \mathbb{S}_+(n): \mathbf{t}^\top X\mathbf{t}\geq 0\ \forall \mathbf{t} \in \mathbb R^n_+\}$$
 and by $\mathbb{CP}(n)$ its dual, the cone of completely positive  matrices,
 $$\mathbb{CP}(n):= {\rm cone}\{\mathbf{x}\mathbf{x}^{\top}, \mathbf{x}\in \mathbb{R}^n_+\}.$$
It is easy to see that the cone $\cop$ can be equivalently defined as
\begin{equation} \cop:=\{X\in \mathbb{S}_+(n): \mathbf{t}^\top X\mathbf{t}\geq 0\ \forall \mathbf{t} \in T \},\label{cop}\end{equation}
where $T:=\{ \mathbf{t} \in  \mathbb R^n_+: ||\mathbf{t}||_1=0\}$.

 Given  an $n\times m$  matrix $A$,  denote by  $A_{ij},$ $i\in [n],$ $j\in [m],$ its  entries.

Let $A\in {\mathbb{ S}}(n)$. The  {\it symmetric vectorization} of $A$, denoted by
$\operatorname{svec}(A) \in \mathbb{R}^{n(n+1)/2}$,
is defined by stacking the upper-triangular entries of $A$, including the diagonal, with off-diagonal entries scaled by $\sqrt{2}$:
$$
\operatorname{svec}(A)
=
\bigl(
A_{11},\,
\sqrt{2}A_{12},\, \ldots, \, \sqrt{2}A_{1n},\,
A_{22},\, \sqrt{2}A_{23}, \, \ldots, \, A_{nn}
\bigr)^\top.
$$
This scaling ensures that the inner product is preserved:
\[
\langle A, B \rangle := \operatorname{trace}(AB) =
\langle \operatorname{svec}(A), \operatorname{svec}(B) \rangle,
\quad \text{for all } A, B \in \mathbb{S}^n.
\]

For a finite family of symmetric $n\times n $ matrices $A(i),$ $i \in {\cal I},$ we define
$ {\rm rank }(A(i),\, i \in {\cal  I}):=$ ${\rm rank }({\rm svec}(A(i)),\, i \in {\cal I}).$

\vspace{2mm}

Denote by ${\mathbb{F}}_{max}(n)$ the set of all maximal faces of the cone $\cp.$
In what follows, we denote by $low(n)$ the  {\it  tight  lower bound} on the dimensions of the faces over all  maximal faces of $\cp$, that is
$$low(n):=\min\limits_{\mathcal{F}\in {\mathbb{F}}_{min}(n)}{\rm dim}(\mathcal{F}).$$

\begin{definition} Let  $ A\in \cop$.
\begin{itemize} \item A nonzero vector $\btau \in {\mathbb R}^n_+$
 is called a {\it zero} of $A $ if $\btau^\top A\btau = 0.$

 \vspace{-2mm}

 \item A zero $\btau$  of $ A $ is called {\it  minimal} if there does not exist another zero $ \bt$ of $A $ such that ${\rm supp}(\bt)$  is
a  proper  subset of ${\rm supp }(\btau)$.\end{itemize}
\end{definition}

In what follows, a zero $\btau$   will be called {\it normalized} if $||\btau||_1=1.$  Due to (\ref{cop}), in what follows,
  without loss of generality,  we will consider only normalized zeros of copositive matrices.

  Given $A\in \cop$, we  denote  by $Z(A)$ the set of all its normalized zeros:
$$Z(A):=\{\btau \in {\mathbb{R}}^n_+:\, \btau^\top A\btau=0,\ ||\btau||_1=1\}.$$
Suppose that $Z(A)\not=\emptyset$  and let  $ \btau^j, j\in [p], $  be the {\it normalized minimal zeros} of the matrix $A$, indexed from $1$ to $p$.
 Denote{:}
$${Z}_{min}(A):=\{\btau^j, j \in J\}  \mbox{ with } J:=[p],$$
and construct the set of index pairs
\begin{equation} E=\{(i,j): \ i \in J, \ j \in J,\ i< j,\ (\btau^i)^\top {A}\btau^j=0\}.\label{N2-1}\end{equation}
  {We then define} the corresponding  undirected  graph $G={G(A):=}(J, E)$  with   {vertex set}  $J$   and  {edge} set   $E$.
	{This graph was introduced in  \cite{KT25}  as  the {\it minimal zeros graph} of the matrix $A$,  and the }
	following lemma was proved  there.

\begin{lemma} \label{R-L-1} For $A\in \cop$ with a nonempty set of normalized {minimal} zeros
{${Z}_{min}(A):=\{\btau^j, j \in J\}$}, let $\{J(s),\ s \in S\}$   denote the set of all maximal ({distinct})
 cliques of the   associated  minimal zeros graph  $G$. Then the  set
 {of all normalized zeros}  $ \  Z(A)$  admits the {following} representation:
\begin{equation} Z(A)=\bigcup\limits_{s \in S}{\cal Z}(s,A), \mbox{ where } {\cal Z}(s,A):={\rm conv} \{{\btau}^j, j \in J(s)\}
\ \forall s \in S.\label{part}\end{equation}
\end{lemma}

 {Note that}  representation (\ref {part})  was {used} in \cite{INFOR} to show  that
\bea&\qquad \qquad {\rm rank}\bigl(\btau \btau^{\top}, \ \btau \in Z(A)\bigl)\nonumber\\
&={\rm rank}\bigl(({\btau}^i+{\btau}^j)({\btau}^i+{\btau}^j)^\top,\ (i,j)\in V(s),\,s \in S\bigl),\nonumber\eea%\end{split}\end{equation}*$$%\label{rank}\ee
where $V(s):=\{(i,j): \ i\in J(s),\, j\in J(s),\, i\leq j\},$ $ s\in S.$ Based on this observation, we obtain the following corollary.

\begin{corollary}\label{OO_L} Consider a matrix $A\in \cop$, the corresponding set of its normalized minimal zeros  $Z_{min}(A)$, the {minimal zeros} graph  $G$, and the set $\{J(s), s \in S\}$ of its maximal cliques.  Then the dimension of the face $\mathcal F:=\cp\cap A^\bot$ of the cone $\cp$ can be determined as follows{:}
$${\rm dim}(\mathcal F)={\rm rank}(({\btau}^i+{\btau}^j)({\btau}^i
+{\btau}^j)^\top,\ (i,j)\in V(s),\,s \in S).$$\end{corollary}

In \cite{local},  given  a copositive matrix $A$ and  one of   its normalized  zeros $\btau$,
 the  set  ${\cal J}(\btau, A)$ was  defined as  the union of ${\rm supp} (\bt)$ over all $\bf t \in Z(A)$   satisfying ${\rm supp}(\btau) \subset {\rm supp}(\bt).$

Representation (\ref{part}) of the set of zeros of $A$  enables  us to  derive  an explicit
  description  of the  set ${\cal J}(\btau, A)$. To  this  end,   we introduce   several  sets and vectors  associated with the  given $A\in \cop$ and $\btau\in Z(A)$, using the corresponding   set of normalized   minimal zeros  $ {\{}{\btau}^j,j \in J {\}},$  and the  family of index  set $\{J(s), s \in S\}$ defined above.  Consider the sets
$$P_*(s):=\bigcup\limits_{j \in J(s)}{\rm supp}({\btau}^j) \  \forall s \in S \mbox{ and } S(\btau):=\{s \in S:\btau\in {\mathcal{Z}}(s,A)\},   $$
and denote $\bt(s):=\sum\limits_{j \in J(s)}{\btau}^j \ \forall s \in S.$

By construction,
\begin{equation} {\rm supp}(\bt(s))=P_*(s) \ \ \forall s \in S,\label{t-s}\end{equation}
and it follows from    Proposition 2  in \cite{KT25} that
\begin{equation} S(\btau)=\{s\in S: {\rm supp}(\btau) \subset P_*(s)\}.\label{Stau}\end{equation}

\begin{proposition}\label{P_02_02}   Assume above notations. Let $A\in \cop$ and let $\btau$ be a normalized zero of  $A$.
  Then the set ${\cal J}(\btau, A)$    admits a representation
\begin{equation} {\cal J}(\btau, A)=\bigcup\limits_{s \in S(\btau)}P_*(s)=\bigcup\limits_{s \in S(\btau)}{\rm supp}({\bf t}(s)).\label{30-1}\end{equation}
\end{proposition}
{\bf Proof.}  Consider a matrix $A\in \cop$ and its zero $\btau\in Z(A).$
It is evident that, due to (\ref{t-s}),  the inclusion ${\rm supp}(\btau)\subset P_*(s)$ implies  ${\rm supp}(\btau)\subset{\rm supp}(\bt(s))$, and consequently,  it holds
$\bigcup\limits_{s \in S(\btau)}{\rm supp}(\bt(s))\subset {\cal J}(\btau, A).$

Suppose  now that there exists $k_0\in {\cal J}(\btau, A)$ such that $k_0\not\in \bigcup\limits_{s \in S(\btau)}{\rm supp}(\bt(s)).$
Since $k_0\in {\cal J}(\btau, A)$, it follows from the definition of the set ${\cal J}(\btau, A)$ that there exists $\bar \bt \in Z(A)$ such that
$$k_0\in {\rm supp}(\bar \bt),\ \ {\rm supp}(\btau)\subset {\rm supp}(\bar \bt).$$  Since $\bar \bt \in Z(A)$,   relation  (\ref{part})
  implies  that there exists $s_0\in S$ such that $\bar \bt\in Z(s_0,A)$. It  then  follows from  Corollary 4 in \cite{KT25} that
$$ {\rm supp}(\bar \bt)\subset P_*(s_0)= {\rm supp}(\bt(s_0)),$$
 and hence
$${\rm supp}(\btau)\subset {\rm supp}(\bar \bt)\subset P_*(s_0).$$
Then, by the definition of the set $S(\btau)$, we   conclude that   $s_0\in S(\btau).$ Consequently,
$${\rm supp}(\bar \bt)\subset \bigcup\limits_{s \in S(\btau)}{\rm supp}(\bt(s)),  $$
 and therefore
$$k_0\in {\rm supp}(\bar \bt)\subset \bigcup\limits_{s \in S(\btau)}{\rm supp}(\bt(s)).$$
But this contradicts the assumption that $k_0\not\in \bigcup\limits_{s \in S(\btau)}{\rm supp}(\bt(s)).$
Thus, we have shown that ${\cal J}(\btau, A)\subset \bigcup\limits_{s \in S(\btau)}{\rm supp}(\bt(s)).$
 Combining this with the  inclusion
$\bigcup\limits_{s \in S(\btau)}{\rm supp}(\bt(s))\subset {\cal J}(\btau, A)$ proved above, we conclude  that the representation
(\ref{30-1})  holds. $\ \Box$

\vspace{4mm}

 Note that for a minimal zero ${\btau}^j$, $j\in J,$ the corresponding set ${\cal J}({\btau}^j, A)$ takes the form
\begin{equation}{\cal J}({\btau}^j, A)=\bigcup\limits_{s\in \{k\in S:j\in J(k)\}} P_*(s).
\label{calJ}\end{equation}

In  \cite{KT-Mathematics}, for  each  $j \in J, $ the set $\bigcup\limits_{s\in \{k\in S:j\in J(k)\}} P_*(s)$ was denoted
 by  $M_*(j)$, and the set $[n]\setminus {\rm supp}(A{\btau}^j)$  was denoted by  $M(j)$:
\begin{equation}\label{Mj} M(j):=[n]\setminus {\rm supp}(A{\btau}^j).\end{equation}

Note that the statements of Corollary \ref{OO_L} and Proposition \ref{P_02_02}  are derived from representation (\ref{part}).
%which once again  highlights its usefulness.

\vspace{3mm}

  It is natural to conjecture  that the    the minimum dimension
of a maximal face of the cone of completely positive $n\times n$ matrices   equals $n(n-1)/2$,
 since this holds for the  cone  $\cop$  and for  the  cone ${\mathbb{CP}}(n)$ with $n\leq 4.$
    Dickinson    refuted this   in  \cite{Dickinson}  by
 exhibiting a maximal face of   $\mathbb{COP}(9)$  of  dimension   $27$.
 Very  recently,   Holmgren and Zhang   (\cite{Zhang})  obtained  the following results.

\begin{lemma}\label{L2} [\cite{Zhang}, Lemma 3.1] For $n>2,$ the dimension of a maximal face of the  cone  $\cp$ is greater than or equal to $n.$
\end{lemma}

\begin{theorem}\label{T4} [\cite{Zhang}, Theorems 4.6 and 4.9] The tight
 lower bound on the dimensions of maximal faces of the  cone   ${\mathbb{CP}}(5)$ is equal to 5.
The  tight   lower bound $low(n) $ on the dimensions of maximal faces of the  cone  $\cp$ is  situated
 between $estim_*(n):=n$ and $estim^*(n):=(n^2 -5n+8)/2$ for $n\geq 6.$
 \end{theorem}

At the moment, these are the    sharpest bounds reported  in the literature.

\vspace{3mm}

The main  aim   of this paper is to improve the above estimates for $low(n).$ Our approach distinguishes between the cases where $n$ is odd and where it is even. %are Theorems \ref{T5} and \ref{Teven} proved in the next sections.

\section{The  Tight Lower Bound on  the Dimensions of Maximal Faces of the Cone $\cp$
of   Odd-Order Matrices}\label{S.3}

{Before presenting the main result of this section, let us recall some definitions that will be needed.}

{An $n\times n$  matrix $A$  is {\it circulant} (see, for example \cite{Tsitsas}) if  each its row is obtained from the previous one by a cyclic shift, {\it i.e. } there exists a vector
$
 \textbf a=(a_0, a_1, \ldots, a_{n-1})^{\top} \in \mathbb R^n,
$
such that
\[
A =
\begin{pmatrix}
a_0     & a_1     & a_2     & \cdots & a_{n-1} \\
a_{n-1} & a_0     & a_1     & \cdots & a_{n-2} \\
\vdots  & \vdots  & \vdots  & \ddots & \vdots  \\
a_1     & a_2     & a_3     & \cdots & a_0
\end{pmatrix}.
\]
In this case, $A$ is completely determined by its first row and is invariant under cyclic permutations.
}

{A vector $\textbf{u} = (u_1, u_2, \ldots, u_n)^{\top} \in \mathbb{R}^n$ is called  {\it palindromic} if
\[
u_i = u_{n+1-i} \quad \text{for all } i \in[n].
\]}

The following theorem is the first  main result of this paper.
\begin{theorem}\label{T5} For every odd $n\geq 5,$ the   tight    lower bound    $low(n) $
   on the dimensions of the maximal faces of the cone $\cp$ is equal to $n$.
 \end{theorem}
 {\bf Proof.} Given an odd $n\geq 5,$ following paper \cite{Hild}, let us  consider a symmetric
  matrix ${\mathcal{A}}\in \mathbb{S}(n)$  whose entries are constructed  according to the following rules{:}
\begin{equation}{\mathcal{A}}_{ij}=\left\{ \begin{array}{ll}
\alpha & \mbox{ if } i=j,\cr
\beta & \mbox{ if } |i-j|\in {\{1,n-1\}},\cr
1& \mbox{ if } |i-j|\in \{2,n-2\},\cr
0& \mbox{ if } |i-j|\in \{3,...,n-3\},\end{array}\right. \ i,j\ \in [n],\label{AA}\end{equation}
where $\alpha:=2(1+2\cos\frac{\pi}{n+1}\cos\frac{3\pi}{n+1})$, $\beta:=-2(\cos\frac{\pi}{n+1}+\cos\frac{3\pi}{n+1})$.

    \vspace{2mm}

For example, for $n=7,$
 matrix ${\mathcal{A}}$ has the form
$${\mathcal{A}}=\begin{pmatrix}
\alpha & \beta & 1 & 0 &  0& 1 & \beta\cr
\beta &\alpha & \beta & 1 & 0 & 0& 1 \cr
1 &\beta &\alpha & \beta & 1 & 0 & 0\cr
0&1 &\beta &\alpha & \beta & 1  & 0\cr
0& 0&1 &\beta &\alpha & \beta & 1 \cr
1&0& 0&1 &\beta &\alpha & \beta \cr
 \beta  &1&0& 0&1 &\beta &\alpha\cr
\end{pmatrix}.$$

{Consider} the following ordered index sets  of cardinality $n -2$:
\begin{equation} I(i):=[n]\setminus \{  n-i,\, n-i+1\} \mbox{ for  }\ i\in [n-1] \mbox{  and  } I(n):=[n]\setminus\{ 1,\, n\}.\label{sets}\end{equation}

It was shown in \cite{Hild} that
\begin{itemize}
\item the matrix ${\mathcal{A}}$ is a  circulant  extremal    copositive matrix;

    \vspace{-3mm}

\item the set of zeros of matrix ${\mathcal{A}}$ consists of $n$ minimal zeros ${\btau}^j$, $j \in [n]$;

    \vspace{-3mm}

\item there exists  a positive palindromic vector  ${\bf u} \in {\mathbb R}^{n-2}_{++}  $
  such that $${\rm  supp}( {\btau}^j) = I(j) \mbox{  and } (\tau^j_k, k \in I(j))^\top={\bf u} \mbox{ for all } j  {\in [n]}.$$
\end{itemize}

Since the matrix ${\mathcal{A}}$ is extremal  and its  set of zeros  consists of $n$ minimal zeros ${\btau}^j$, $j\in [n]$,
we can conclude that
\begin{equation}{\rm rank}({\btau}^j, \,j\in [n])=n.\label{eq1}\end{equation}

Let us show that the matrix ${\mathcal{A}}$ generates an  exposed ray  of the cone $ \cop$.
 It follows from   \cite{KT-Mathematics}  that to  this end,  it is  enough  to show  that
\begin{equation} {\rm supp}({\btau}^j)=[n]\setminus {\rm supp}({\mathcal{A}}{\btau}^j) \ \forall j\in [n].\label{eq2}\end{equation}

  Since  $n$ is odd,   we can write   $n-2=2q+1$   for some  $q\geq 1.$ Then the vector ${\bf u}$  introduced above
takes the form  ${\bf u}=(u_1, u_2, ...,u_{n-2})^\top > \textbf{0},$ where
\begin{equation} u_k=u_{2q-k+2} >0 \ \forall k  \in [q]  ,\ \ u_{q+1}>0.\label{eq3}\end{equation}
By construction,   given  $j\in [n] $ and  the corresponding zero ${\btau}^j,$ we have
\begin{equation} (\tau^j_k, k \in I(j))^\top={\bf u}, \ \tau^j_k=0 \mbox{ for } k \in [n]\setminus I(j).\label{23_1}\end{equation}
Since ${\btau}^j$ is a zero of the matrix ${\mathcal{A}}$, the following relations hold true
$${\mathbf{e}_i^\top \mathcal{A}\btau^j = 0}
\quad \forall i \in I(j), \quad
{\mathbf{e}_i^\top \mathcal{A}\btau^j = \lambda_i \ge 0}
\quad \forall i \in [n]\setminus I(j).$$
It follows from these relations and (\ref{23_1}) that
\bea& \sum\limits_{k\in I(j)}\mathcal{A}_{ik}u_k=0 \ \forall i \in I(j),\label{eq4}\\
&\sum\limits_{k\in I(j)}\mathcal{A}_{ik}u_k=:\lambda_i \geq 0\ \forall i \in [n]\setminus I(j).\label{eq5}\eea
In view of the specific structure of the matrix $\mathcal{A}$ and  relations (\ref{eq3}),  we readily obtain
$\lambda_n=\lambda_{n-1}$.

Let us show that $\lambda_n=\lambda_{n-1}>0.$  Suppose, to the contrary, that $\lambda_n=\lambda_{n-1}=0.$
Taking into account these equalities and  summing {equalities} in (\ref{eq4}), (\ref{eq5}),
we obtain
$$(2 +2\beta+\alpha)(u_1+u_2+...+u_{n-2})=0, \mbox{ where }  u_1+u_2+...+u_{n-2}>0$$
 and
$$2 +2\beta+\alpha=4(1-\cos(\frac{\pi}{n+1}))(1-\cos(\frac{3\pi}{n+1}))>0.$$
The contradiction obtained proves  that $\ \lambda_n=\lambda_{n-1}>0.\ $  Consequently,\\
${\rm supp}({\mathcal{A}}{\btau}^j)=[n]\setminus I(j)$ and, by construction,  we have
that ${\rm supp}({\btau}^j)=I(j).$ Thus, the  equalities  (\ref{eq2}) hold true.

Hence, we have shown that the matrix ${\mathcal{A}}$ {generates }  an exposed ray of the cone $\cop$.
 {Then} it follows from Theorem \ref{T1} that the set ${\mathcal F}:=\cp\cap {\mathcal{A}}^\bot$ is a maximal face of the cone $\cp.$
This face  can be presented in the form
$$\mathcal F=\{B=\sum\limits_{j=1}^n\gamma_j{\btau}^j({\btau}^j)^\top, \ \gamma_j\geq 0 \ \forall j\in [n]\}.$$
It follows from th{e} representation {above} and equality (\ref{eq1}) that
$${\rm dim}(\mathcal F)={\rm dim}({\rm span}\{{\btau}^j({\btau}^j)^\top, \  j\in [n]\})=n.$$

 Taking into account Lemma \ref{L2}, we conclude that  for an  odd $n\geq 5$,   it holds $ low(n)=n$. \ \ \ $\Box$

\vspace{3mm}

  Therefore,  for any odd integer $n\geq 5,$ we have  determined   the exact value for  the  tight  lower bound on the dimensions of maximal faces of the cone $\cp$.

\section{Estimation of the Maximal Face Dimensions for the Cone $\cop$ of
Even-Order Matrices}\label{S.4}
\subsection{Preliminary Results}

In this subsection, we introduce a special set of copositive matrices from the cone $\mathbb{COP}(n+1)$ constructed on the basis of a given extremal matrix $A \in \cop$.
We derive several properties of these matrices in $\mathbb{COP}(n+1)$ and of their associated sets of minimal zeros.

\begin{proposition}\label{19P_1}
Let a matrix $A\in \cop$ and a set $I\subset [n]$, $I\not=\emptyset,$  be given.
Denote ${\bf{e_*}}:=\sum\limits_{i \in I}{\bf{e}}_i\in \mathbb R^n_+ $ and
\begin{equation} B=B(A,I):=\begin{pmatrix}
A&A{\bf{e_*}}\cr
{\bf{e}}^\top_* A & \,{\bf{e}}^\top_* A{\bf{e_*}}\end{pmatrix}.\label{B_matrix}\end{equation}
Then $B\in {\mathbb{COP}}(n+1).$
\end{proposition}
{\bf Proof.} Consider a vector  $\bar \bt=(\bt^\top, t_0)^\top\in \mathbb R^{n+1}_+,$ where $ \bt \in \mathbb R^n_+$ and $t_0\in \mathbb R_+.$
Then
\begin{equation} \bar \bt^\top B\bar \bt=\bt^\top A \bt+2\bt^\top A{\bf{e_*}}t_0 +{\bf{e_*}}^{\top} A{\bf{e_*}}t^2_0=
 (\bt+t_0{\bf{e_*}})^\top A (\bt+t_0{\bf{e_*}}),\label{19_1}\end{equation}
and, taking into account that $A\in \cop$ and $(\bt+t_0{\bf{e_*}})\in {\mathbb{ R}}^n_+,$ we  can conclude that
 $${\bar \bt}^\top B\bar \bt=(\bt+t_0{\bf{e_*}})^\top A (\bt+t_0{\bf{e_*}})\geq 0.$$

Thus, we have shown that for the matrix $B$ in the form (\ref{B_matrix}), it holds
$B\in {\mathbb{COP}}(n+1).$ $\ \Box$

\vspace{2mm}

Note that it is easy to see that Proposition \ref {19P_1}  holds true if we replace $\bf{e_*}$ by any vector ${\bf{a_*}}\in \mathbb R^n_+.$

\vspace{2mm}

Let $Z_{min}(A)=\{{\btau}^j, j \in J\} $  be the set of all normalized minimal zeros  of
 a copositive matrix $A\in \cop$  and let $I\subset [n]$, $I\not=\emptyset.$
Denote
\begin{equation}\label{**J}J_{0}= J_0(A,I):=\{j \in J:I\subset {\rm supp}({\btau}^j)\}.\end{equation}
\begin{proposition}\label{P3}  Consider  a matrix  $A\in \cop$  with  a finite set of normalized zeros,  that is
 $ Z(A)= Z_{min}(A)= \{ {\btau}^j, j \in J \}.$
Let  the sets  $I$ and $J_{0}$, the vector  $\bf{e_*}$,   and  the matrix  $B=B(A,I)$ be defined  as above. Then the vectors
\begin{equation}\bar {\btau}^j{:}=\begin{pmatrix}
{\btau}^j\cr
0\end{pmatrix},   j \in J, \mbox{ and }
 {\bar \by}^j:=  \frac{1}{\mu_j} \begin{pmatrix}
{\btau}^j-\sigma_j\bf{e_*}\cr
\sigma_j\end{pmatrix},  j \in J_{0},\label{zerotau}\end{equation}
 where
$$ \sigma_j:=\min\{\tau^j_k, k \in I\}>0,\
\mu_j:=1-\sigma_{j}(|I|-1)>0 \ \mbox{ for } j \in J_{0},$$
are  the normalized minimal  zeros of  the  matrix $B$ and the set  of all normalized zeros of $B$  has the form
\begin{equation} Z(B)=\{\bar {\btau}^j, j \in J\setminus J_{0}\}\bigcup\limits_{j \in J_0} {\rm conv}\{\bar {\btau}^j, {\bar \by}^j\}.\label{zeroset}\end{equation}
\end{proposition}

{\bf Proof.} It is easy to see that by construction,  the  vectors $\bar {\btau}^j,$ $ \in J, $ defined in (\ref{zerotau}),  are normalized zeros of the matrix $B.$

Suppose that $J_{0} \neq \emptyset$, and consider  the  vectors ${\bar \by}^j,$ $j\in J_0,$ defined in  (\ref{zerotau}).
For  each  $j\in J_{0},$ due to the special choice of $\sigma_j$ we have  ${\bar \by}^j \geq  {\bf 0}$, and due to the special choice of $\mu_j$ we have $||{\bar \by}^j||_1=1.$
 It follows from (\ref{19_1}) that
 $$({\bar \by}^j)^\top B {\bar \by}^j=(1/\mu_j^2)({\btau}^j-\sigma_j {\bf{e_*}}+\sigma_j {\bf{e_*}})^\top A ({\btau}^j-\sigma_j{\bf{e_*}}+\sigma_j {\bf{e_*}})$$
$${=}(1/\mu_j^2)({\btau}^j)^\top A{\btau}^j=0 \ \forall j \in J_{0}.$$
 Consequently, vectors ${\bar \by}^j,$ $j\in J_0,$ are normalized zeros of $B$.
  Thus, we have shown that all vectors in  (\ref{zerotau}) are normalized zeros of the matrix $B$.

Let us show that for any $j\in J_0$ and any $\alpha\in [0,1],$ the  vector
$$\bt(j,\alpha):=\alpha\bar{\btau}^j+(1-\alpha){\bar \by}^j$$ is a normalized zero of  the matrix   $B$.
Notice that by construction, $\bt(j,\alpha) {\geq} {\bf 0}$ and $||\bt(j,\alpha)||_1=1.$
Let us calculate
$$(\bt(j,\alpha))^\top B\bt(j,\alpha)=2\alpha(1-\alpha)(\bar{\btau}^j)^\top B{\bar \by}^j$$
$$=\frac{2\alpha(1-\alpha)}{\mu_j}{\begin{pmatrix}
{\btau}^j\cr
0\end{pmatrix}^\top}
\begin{pmatrix}
A&A\bf{e_*}\cr
{\bf{e}}^\top_* A&{\bf{e}}^\top_* A\bf{e_*}\end{pmatrix}\begin{pmatrix}
{\btau}^j-\sigma_j\bf{e_*}\cr
\sigma_j\end{pmatrix}
 {=}\frac{2\alpha(1-\alpha)}{\mu_j}({\btau}^j)^\top A{\btau}^j=0.$$
Hence, we have shown that
$${\cal T}:=\{\bar {\btau}^j, j \in J\setminus   {J_0}\}\bigcup\limits_{j \in   {J_0}} {\rm conv}\{\bar {\btau}^j, {\bar \by}^j\}\subset Z(B).$$

Now let us consider any vector $\bar \bz=(\bz^\top, z_0)^\top \in Z(B).$   {From the definition of $\bar \bz$ and $B$, it follows} that $||\bz||_1+z_0=1$ and
${\tilde \bz}:=(1/\mu)(\bz+z_0{\bf{e_*}})\in Z(A)$ with
$$\mu:=|| \bz+ z_0 {\bf{e_*}}||_{1}=1+z_0(|I|-1).$$
Taking into account that $Z(A)= {Z_{min}(A)}=\{{\btau}^j, j \in J\}$, we conclude that there exists an index  $j_*\in J$ such that
$\tilde \bz=\btau^{j_*}$ and, consequently,  $\bz=\mu\btau^{j_*}-z_0\bf{e_*}$.
Suppose  that  $z_0=0$. Then $\bar \bz=((\btau^{j_*})^\top, 0)^\top=\bar{\btau}^{j_*}$ and   hence   $\bar \bz\in {\cal T}.$

Now suppose that $z_0>0$. From  the relations  $\bz=\mu\btau^{j_*}-z_0e_*$ and
 $\bz\geq{\bf 0}${, it follows} that
$I\subset {\rm supp}(\btau^{j_*})$ and therefore $j_*\in J_0$ and $0<z_0\leq \mu\sigma_{j_*}.$
This implies that $\bar \bz=((\mu\btau^{j_*}-z_0{\bf{e_*}})^\top, z_0)^\top.$

It is not difficult to check that
$$\bar \bz=\begin{pmatrix}
\mu\btau^{j_*}-z_0\bf{e_*}\cr
z_0\end{pmatrix}=\alpha_*\bar \btau^{j_*}+(1-\alpha_*){\bar \by}^{j_*} \mbox{ with } \alpha_*:=1-\frac{z_0\mu_{j_*}}{\sigma_{j_*}}.$$
 Since $0<z_0\leq\mu\sigma_{j_*}, $ $\mu_{j_*}>0$, and $\sigma_{j_*}>0$, it  is   {easy to show}  that $\alpha_*\in [0,1).$ Consequently,
 $\bar \bz\in {\rm conv}\{\bar\btau^{j_*},{\bar \by}^{j_*}\}\subset {\cal T}.$

Thus, we have shown that for any $\bar \bz\in Z(B)$,  the inclusion   $\bar \bz\in {\cal T}$ holds, and   hence  $Z(B)\subset {\cal T}.$
 Combining  this  with  the  inclusion ${\cal T}\subset Z(B)$ proved above, we obtain   $Z(B)= {\cal T}.$  Therefore, equality (\ref{zeroset}) holds true.

Under   the stated  assumptions, it is evident that  the  vectors $\bar {\btau}^j,$ $ \in J, $
defined in (\ref{zerotau}) are normalized minimal zeros of the matrix $B$.

Let us show that the vectors $ {\bar \by}^j$, $ j\in J_0,$ are minimal zeros of the matrix $B$ as well.  To do this, let us introduce index sets
$$\Delta I(j):={\rm supp}({\btau}^j)\setminus I,\ Q(j):=\{k \in I:\tau^j_k>\sigma_j\} \ \forall j \in J_0.$$
Since $I\subset {\rm supp}({\btau}^j)$ for all $j \in J_0$, we have
\begin{equation} {\rm supp}({\btau}^j)=\Delta I(j)\cup I \ \ \forall j \in J_0,\label{12-1}\end{equation}
and it is easy to see that
\begin{equation} {\rm supp}({\btau}^j-\sigma_j{\bf{e_*}})=Q(j)\cup \Delta I(j) \ \, \forall j \in J_0.\label{12-2}\end{equation}
Note that by construction,
\begin{equation} Q(j)\cap \Delta I(i) =\emptyset \ \, \forall j \in J_0, \ \, \forall i \in J_0.\label{12-3}\end{equation}

Suppose that there exists $j_0\in J_0$ such that the vector ${\bar \by}^{j_0}\in Z(B)$ is not a minimal zero of the matrix $B$. Then there exists $\bar \bt \in Z(B)$, $\bar \bt\not={\bar \by}^{j_0},$
such that
\begin{equation} {\rm supp}(\bar \bt)\subset {\rm supp}({\bar \by}^{j_0}).\label{12-4}\end{equation}
Since $\bar \bt\in Z(B)$, it follows from (\ref{zeroset}) that  one of  the following  cases  may  occur:
\begin{itemize}
\item[(A)] There exists $j_*\in J$ such that $\bar \bt=\bar\btau^{j_*}.$

\vspace{-3mm}

\item[(B)]  There exists $j_*\in J_0\setminus j_0$ such that $\bar \bt= {\bar \by}^{j_*}$.

\vspace{-3mm}

\item[(C)]   There exist $j_*\in J_0$ and $\alpha\in (0,1)$ such that $\bar \bt=\alpha\bar\btau^{j_*}+(1-\alpha) {\bar \by}^{j_*}.$
\end{itemize}

 We consider these cases separately. In  case (A), we have  $\bar \bt=\bar\btau^{j_*}$. Hence it follows from (\ref{12-4}) that
\begin{equation} \begin{split}&{\rm supp}({\btau}^{j_*})\subset {\rm supp}({\btau}^{j_0}-\sigma_{j_0}{\bf{e_*}})\subset {\rm supp}({\btau}^{j_0}),\\
&{\rm supp}({\btau}^{j_0}-\sigma_{j_0}{\bf{e_*}})\neq {\rm supp}({\btau}^{j_0}).\end{split}\label{13-1}\end{equation}
This contradicts the assumption that ${\btau}^{j_0} $ is a minimal zero of the matrix $A$.  Therefore, case (A) is not possible.

Suppose that case    (B)  occurs. Then, from  (\ref{12-4})  we deduce  that
$${\rm supp}(\btau^{j_*}-\sigma_{j_*}{\bf{e_*}})\subset {\rm supp}({\btau}^{j_0}-\sigma_{j_0}{\bf{e_*}}), \ \mbox{ where } j_*\not=j_0,$$
 and, taking into account (\ref{12-2}), we obtain
$$Q(j_*)\cup \Delta I(j_*)\subset Q(j_0)\cup \Delta I(j_0). $$
This inclusion, together with  (\ref{12-3}) impies  the inclusion $\Delta I(j_*)\subset \Delta I(j_0)$,   from which  and (\ref{12-1}) we obtain
${\rm supp}({\btau}^{j_*})\subset {\rm supp}({\btau}^{j_0}).$ This again contradicts the fact that $\btau^{j_0} $ is a minimal zero of the matrix $A$. Hence, case (B) cannot occur.

Consider case (C). We  then  have
$$ {\rm supp}(\bar \bt)={\rm supp}(\bar \btau^{j_*})\cup {\rm supp}({\bar \by}^{j_*})={\rm supp}( \btau^{j_*})\cup\{n+1\}.$$
Consequently, ${\rm supp}({\bf t})={\rm supp}( \btau^{j_*})$ and it follows from (\ref{12-4}) that
relations (\ref{13-1})  hold.  However,   this again contradicts the assumption that $\btau^{j_0} $ is a minimal zero of the matrix $A.$

Thus, we have shown that all vectors in (\ref{zerotau}) are normalized minimal zeros of the matrix $B$.

It follows from (\ref{zeroset})  that  any normalized zero of $B$ can be represented as a convex combination of
some vectors from (\ref{zerotau}).   Hence the matrix $B$ has no other normalized  minimal zeros except the zeros defined in (\ref{zeroset}).
$\  \Box$

\vspace{4mm}

Let  matrices $A$  and $B$, vector  ${\bf{e_*}}$,  and sets $I$ and  $J_0$ be defined  as above.
Consider the set  ${Z_{min}}(A)=\{{\btau}^j, j \in J\} $ of all normalized minimal zeros of the matrix $A$.
For $j \in J,$ since ${\btau}^j$ is a zero of $A$, we have (see, {\it e.g.},  \cite{KT-Mathematics}):
$$\mathbf{e}^\top_kA{\btau}^j=0\ \forall k\in {\rm supp}({\btau}^j);\  \mathbf{e}^\top_kA{\btau}^j\geq 0\ \forall k\in [n]\setminus {\rm supp}({\btau}^j).$$
It follows from these relations that
\begin{equation} {\rm supp}({\btau}^j)\subset  M(j) \ \  \forall j   \in J,\label{20_10}\end{equation}
\begin{equation} {\bf{e}}^\top_* A{\btau}^j=0,\ I\subset M(j )\ \  \forall j \in  J_0,\label{19_3}\end{equation}
where the sets $M(j),$ $j \in J,$ are defined in (\ref{Mj}).

\begin{proposition}\label{P_extreme} Let $A\in \cop$ be   a matrix whose  set of all normalized zeros consists of a finite number of vectors:
$Z(A)= Z_{min}(A)= \{{\btau}^j, j \in J\}.$    Suppose that $A$ is  an  extremal  {matrix of $\cop$} and
\begin{equation} \bigcup\limits_{j\in J_0}M(j)=[n].\label{20_11}\end{equation}
Then the matrix $B=B(A,I)\in {\mathbb{S}}(n+1)$ is an extremal copositive matrix  of the cone $\mathbb{COP}(n+1).$

\end{proposition}
{\bf Proof.} Notice that it follows from Proposition \ref{19P_1} that $B\in {\mathbb{COP}}(n+1)$ and
the vectors $\bar{\btau}^j, j \in J,$ ${\bar \by}^j,j \in J_0,$ defined in (\ref{zerotau}) form the set of all normalized minimal zeros of the matrix $B$.

By construction,
\begin{equation} B\bar{\btau}^j=\begin{pmatrix}
A{\btau}^j\cr
{\bf{e}}^\top_* A{\btau}^j\end{pmatrix} \ \forall j \in J;\
 B {\bar \by}^j=\frac{1}{\mu_j}\begin{pmatrix}
A{\btau}^j\cr
{\bf{e}}^\top_* A{\btau}^j\end{pmatrix} \ \forall j \in J_0.\label{101}\end{equation}
Hence, taking into account (\ref{19_3}), we have
\begin{eqnarray}  & {\bar {\bf e}}^\top _kB\bar{\btau}^j=0\ \forall k \in M(j),\ \forall j \in J\setminus J_0;\label{21_1}\\
&{\bar {\bf e}}^\top _kB\bar{\btau}^j={\bar {\bf e}}^\top _kB {\bar \by}^j=0\ \forall k \in M(j)\cup\{n+1\},\ \forall j \in J_0,\label{21_2}\end{eqnarray}
where   for $k \in [n+1],$   ${\bar {\bf e}}_k$  denotes the
$k$-th standard basis vector in $\mathbb R^{n+1}.$

Let $\bar D=\begin{pmatrix}
D & {\bf d}\cr
{\bf d}^\top & d_*\end{pmatrix}\in {\mathbb{COP}}(n+1)$ and $\bar C=\begin{pmatrix}
C & \bc \cr
\bc^\top & c_*\end{pmatrix}\in {\mathbb{COP}}(n+1)$, where  the matrices  $D\in {\mathbb{S}}(n)$, $C\in {\mathbb{S}}(n)$,
 and the vectors $\bc\in \mathbb R^{n}$ and $\bd\in \mathbb R^{n}$,  be such that
$$B=\begin{pmatrix}
A & A\bf{e}_*\cr
{\bf{e}}^\top_*  A &\ {\bf{e}}^\top_* A\bf{e}_*\end{pmatrix}=\bar D+\Bar C,$$
or equivalently
$$ A=D+C,\ \bd+\bc=A{\bf{e}_*},\ d_*+c_*={\bf{e}}^\top_*A{\bf{e}}_*.$$
Since $D \in {\mathbb{COP}}(n),$ $C \in {\mathbb{COP}}(n),$ and, by assumption,  $A$ is an extremal
copositive matrix,   it follows from the equality
$A=D+C$ that
$$D=\alpha A, \ C=(1-\alpha)A \ \mbox{ with some } \ \alpha\in [0,1].$$

From the equality $B=\bar D+\bar C$ and  the inclusions $B\in {\mathbb{COP}}(n+1)$,
$\bar D\in {\mathbb{COP}}(n+1)$,  {and} $\bar C\in {\mathbb{COP}}(n+1)$, one may conclude that the vectors $\bar{\btau}^j, \ j \in J,$  {and} ${\bar \by}^j, \ j \in J_0,$  must be zeros of  both  matrices  $\bar D$ and $\bar C$. This implies that
$$\mathbf{\bar e}^\top_k \bar D {\bar \by}^j\geq 0,\ \mathbf{\bar e}^\top_k \bar C {\bar \by}^j\geq 0 \ \forall k\in [n+1], \ \forall j \in J_0.$$
Taking into account  {the} inequalities  above, along with the equalities (\ref{21_2}), we obtain
$$\mathbf{\bar e}^\top_k\bar D {\bar \by}^j=0 ,\ \mathbf{\bar e}^\top_k\bar C {\bar \by}^j=0\ \forall k \in  M(j)\cup\{n+1\},\ \forall j \in J_0.$$

  Next, noting  that $\mathbf{e}^\top_k A{\btau}^j=0$, $ k \in {M(j)},$ for all $j \in J_0$ and  using  $D=\alpha A,$  we  calculate
$$\mu_j\mathbf{\bar e}^\top_k\bar D {\bar \by}^j=\mathbf{\bar e}^\top_k
\begin{pmatrix}
\alpha A & \bd\cr
\bd^\top & d_*\end{pmatrix}
\begin{pmatrix}
{\btau}^j-\sigma_j\bf{e}_*\cr
\sigma_j\end{pmatrix}=(\alpha \mathbf{e}^\top_kA, \,\mathbf{e}^\top_k \bd)\begin{pmatrix}
{\btau}^j-\sigma_j\bf{e}_*\cr
\sigma_j\end{pmatrix}$$
\begin{equation} = (-\alpha\sigma_j\mathbf{e}^\top_kA{\bf{e}_*}+\sigma_j\mathbf{e}^\top_k\bd)=0\ \  \forall k \in M(j),\ \forall j \in J_0,\label{21_3}\end{equation}
\begin{equation} \mu_j\mathbf{\bar e}^\top_{n+1}\bar D {\bar \by}^j=(\bd^\top {\btau}^j-\sigma_j\bd^\top {\bf{e}_*}+\sigma_jd_*)=0 \ \forall j \in J_0.\label{21_4}\end{equation}
It then follows from (\ref{21_3}) that
$$\mathbf{e}^\top_k\bd=\alpha \mathbf{e}^\top_kA{\bf{e}_*} \ \; \forall k \in \bigcup\limits_{j\in J_0}M(j).$$
From these equalities and {condition (\ref{20_11})},
 we  conclude  that $\bd=\alpha A{\bf{e}_*},$  and consequently (see (\ref{19_3})), $\bd^\top {\btau}^j=0$ for all
$j\in J_0.$
Then it follows from (\ref{21_4}) that $d_*=\bd^\top {\bf{e}_*}=\alpha \bf{e}^\top_*A{\bf{e}_*}.$
Thus, we have shown that $\bar D=\alpha B.$

In a similar way, one can show that $\bar C=(1-\alpha) B.$ By definition,  {it follows} that $B$ is extremal. $\ \Box$

\begin{proposition}\label{Pcond}  Suppose that a copositive matrix $A\in \cop$ and a set $I\subset [n],$ $I\not=\emptyset,$ be such that the following conditions are satisfied:

a) $Z(A)=Z_{min}(A)=\{{\btau}^j, j \in J\}$;

b) ${\rm supp}({\btau}^j)\cup {\rm supp}(A{\btau}^j)=[n] \ \ \forall j \in J;$

c) matrix $A$  generates  an extreme ray of the cone $\cop;$

d) $\bigcup\limits_{j\in J_0} {\rm supp}({\btau}^j)=[n] $, where $J_0=\{j \in J: I\subset {\rm supp}({\btau}^j)\}.$\\
Then the matrix $B=B(A,I)$ constructed by formula (\ref{B_matrix}),  generates an  extreme exposed ray of the cone ${\mathbb{COP}}(n+1).$
\end{proposition}

{\bf Proof.} It follows from the definition (\ref{Mj})  of the sets $M(j)$, together with condition $b)$,
that $M(j)=[n]\setminus {\rm supp}(A{\btau}^j)={\rm supp}({\btau}^j)$ for all $j\in J.$ Hence, under the assumptions of
this proposition, all the hypotheses of Proposition \ref{P_extreme} are  satisfied.
Consequently, the matrix $B$  generates an extreme ray of  the cone ${\mathbb{COP}}(n+1).$

Note that it follows from condition $b)$ and the definition of the set $J_0$ that
\begin{equation} {\bf{e}}^\top_*A{\btau}^j>0\ \forall j \in J\setminus J_0, \ \ {\bf{e}}^{\top}_*A{\btau}^j=0\ \forall j \in  J_0.\label{100}\end{equation}
It then  follows  from equalities (\ref{101}) and  relations (\ref{100})  that,  for the  matrix $B=B(A,I)$ under consideration and  {the} vectors  $\bar {\btau}^j$, $j\in J, $ and $ {\bar \by}^j, j\in  J_0, $ defined in (\ref{zerotau}), we have
$${\cal J}(\bar {\btau}^j, B)={\rm supp}({\btau}^j),\ {\rm supp}(B\bar{\btau}^j)={\rm supp}(A{\btau}^j)\cup\{n+1\} \ \forall j \in J\setminus J_0;$$
$${\cal J}(\bar {\btau}^j, B)={\cal J}( {\bar \by}^j, B)={\rm supp}({\btau}^j)\cup\{n+1\},$$
$$ {\rm supp}(B\bar{\btau}^j)={\rm supp}(B {\bar \by}^j)={\rm supp}(A{\btau}^j) \ \forall j \in  J_0.$$
Here, for a given copositive matrix $B$ and its minimal zero $\bar \btau $, the  index set $ {\cal J}(\bar \btau , B)$ is defined by
   (\ref{30-1}) (see also  (\ref{calJ})).

It follows from the equalities above and   the assumption   $b)$  of this proposition,   that
\be\begin{split} &{\cal J}(\bar {\btau}^j, B)=[n+1]\setminus {\rm supp}(B\bar{\btau}^j) \ \forall j \in J;\\
&{\cal J}( {\bar \by}^j, B)=[n+1]\setminus {\rm supp}(B {\bar \by}^j) \ \forall j \in J_0.\end{split}\label{102}\end{equation}

Since the matrix $B$ is an extreme ray of the cone ${\mathbb{COP}}(n+1)$, it follows from Theorem 17 in \cite{local} that all solutions $\bar D\in {\mathbb{S}}(n+1)$ of
the following system:
\be\begin{split} & {\bar {\bf  e}}^\top_k \bar D\bar{\btau}^j=0\ \forall k \in [n+1]\setminus {\rm supp}(B\bar {\btau}^j),\ \forall j \in J;\\
& {\bar {\bf  e}}^\top_k \bar D{\bar \by}^j=0\ \forall k \in [n+1]\setminus {\rm supp}(B{\bar \by}^j),\ \forall j \in J_0,\label{104}\end{split}\end{equation}
have the form $\bar D=\alpha B.$

 Theorem 19 in \cite{local}  states  that $B\in{\mathbb{COP}}(n+1)$ is an exposed  ray if  and only if
all solutions $\bar D\in {\mathbb{S}}(n+1)$ of
the  following  system:
\be\begin{split}&  {\bar {\bf  e}}^\top_k \bar D\bar {\btau}^j=0\ \forall k \in {\cal J}(\bar {\btau}^j, B),\  {j \in J};\\
& {\bar {\bf  e}}^\top_k \bar D{\bar \by}^j=0\ \forall k \in {\cal J}({\bar \by}^j, B),\ \forall j \in J_0,\label{106}\end{split}\end{equation}
\be\begin{split}&  {\bar {\bf  e}}^\top_k \bar D\bar {\btau}^j\geq 0\ \forall k \in [n+1]\setminus {\rm supp}(B\bar {\btau}^j)\ \forall j \in J;\\
& {\bar {\bf  e}}^\top_k \bar D{\bar \by}^j\geq 0\ \forall k \in [n+1]\setminus {\rm supp}(B{\bar \by}^j)\ \forall j \in J_0,\end{split}\label{108}\end{equation}
have the form $\bar D=\alpha B.$
Taking into account
equalities (\ref{102}),
 we conclude that  system (\ref{106}),  (\ref{108})  has the same solution set as system
 (\ref{104}),  which  has solutions only in the form $\bar D=\alpha B.$

Thus, we have shown that all solutions of the system   (\ref{106}),  (\ref{108}) have the form $\bar D=\alpha B$,
and consequently, the matrix $B$   generates  an exposed ray of the cone ${\mathbb{COP}}(n+1).$  $\ \Box$

\vspace{2mm}

 The following corollary is  an  evident consequence of Propositions \ref{P_extreme} and \ref{Pcond}.

\begin{corollary}\label{CC2}  Let $A\in \cop$ and $I\subset [n]$ with  $|I|\geq 2$  be  as in Propositions \ref{P_extreme} (Proposition \ref{Pcond}). Then, for any $\bar I\subset I,$ the  matrix $B(A,\bar I)\in \mathbb{COP}$ is extremal (respectively,  exposed).
\end{corollary}

\begin{remark}  In Propositions \ref{P3}-\ref{Pcond}, {the}
assumption that $Z(A)=Z_{min}(A)$ can be relaxed   by using a   corresponding new definition of the set $J_0$.
\end{remark}

If $|I|=1,$ then Propositions \ref{P_extreme} and \ref{Pcond}   follow from  Theorem 3.8 in \cite{Baumert}.
 Note that  for a given extremal/exposed matrix  $A\in \cop$, the application of  these propositions with different sets $I\subset  [n]$,
 {allows us to construct } more {distinct} extremal/exposed matrices in $\mathbb{COP}(n+1)$ in the form $B(A,I)$,
 than by applying Theorem 3.8 in \cite{Baumert}.
 In general,  matrices $B(A,I)$ and  $B(A,\bar I)$ with $|I|\neq |\bar I|$ will   possess different numbers
 of normalized minimal zeros and, consequently,   distinct   sets of normalized zeros.

\subsection{  Estimation  of $low(n)$  for  even $n$ }

Now we are ready to prove the  second  main result of this paper.

\begin{theorem}\label{Teven}
For  an  even $\bar n\geq 6,$  the following estimate holds: $$n\leq low(n) \leq n+3.$$
 \end{theorem}
{\bf Proof.}  Let $\bar n$ be an even integer  number, $\bar n\geq 6.$  Set $n:=\bar n-1$ and consider a
matrix ${\mathcal{A}}\in \cop$  whose  entries  are  defined by  the rules
(\ref{AA}).  As it was shown above, the matrix  ${\mathcal{A}}$   {generates} an exposed ray  of the cone $ \cop$  and
$$Z( \mathcal{A})=Z_{min}(\mathcal{A})=\{{\btau}^j, \, j \in J\} \mbox{ with } J=[n],$$
$${\rm supp}({\btau}^j)=I(j), \ \ {\rm supp}({\btau}^j)\cup {\rm supp}({\mathcal{A}}{\btau}^j)=[n] \  \ \forall j \in [n],$$
with  {the} index sets $I(j)$, $j\in [n],$ defined in (\ref{sets}).

  If we set $I:=\ [n-4]$, then   the set $J_0$ defined in (\ref{**J}) has the form
$$J_0:=\{j\in [n]:I\subset {\rm supp}({\btau}^j)\}=\{1, 2, 3\}$$
 and it is easy to verify that  $\bigcup\limits_{j\in J_0}{\rm supp}({\btau}^j)=[n].$

Thus, we see that, for the  matrix ${\mathcal{A}}$ and  the  set $I$, all  the   assumptions of Proposition \ref{Pcond} are  satisfied.
 Consequently, the corresponding matrix $B=B({\mathcal{A}},I)$  generates  an
exposed ray of the cone ${\mathbb{COP}}(\bar n)$, the vectors $\bar {\btau}^j,j\in J,$ $ {\bar \by}^j, j \in J_0,$ defined in  (\ref{zerotau}), are normalized
 minimal zeros of $B$, and the set of all  normalized zeros $Z(B)$ is  given by  (\ref{zeroset}).

Consider a face ${\cal F}:={\mathbb{CP}}(\bar n)\cap B^\bot$  of the cone ${\mathbb{CP}}(\bar n).$ Since $B$
  generates an exposed ray of ${\mathbb{COP}}(\bar n)$,  Theorem \ref{T1}  implies  that  ${\cal F}$
is a maximal face of the cone ${\mathbb{CP}}(\bar n).$
  {It follows from}  Corollary \ref{OO_L} and  the  representation (\ref{zeroset}) of the set $Z(B)$ that
the dimension of the face ${\cal F}$ is equal to
$$k_0:={\rm rank} \bigl(\bar {\btau}^j(\bar {\btau}^j)^\top \!\!, \, j \in J, \,  {\bar \by}^j({\bar \by}^j)^\top \!\!, \,
(\bar {\btau}^j + {\bar \by}^j)(\bar {\btau}^j + {\bar \by}^j)^\top \!\!, \, j \in J_0\bigl).$$
Taking into account the specific structure of the vectors $\bar {\btau}^j, j \in J,$ ${\bar \by}^j, j \in J_0,$ and equality (\ref{eq1}),
 it is easy to see that
\begin{equation*}\begin{split}k_0=&{\rm rank}\bigl(\bar {\btau}^j(\bar {\btau}^j)^\top \!, \, j \in J, \,(\bar {\btau}^j{\bar {\bf e}}^\top_*+\bar {\bf e}_*(\bar{\btau}^j)^\top),\, j \in J_0,\
\bar {\bf e}_*{\bar {\bf e}}^\top_*\bigl)\\
=&|J|+|J_0|+1=n+3+1=\bar n+3,\end{split}\end{equation*}
where ${\bar {\bf e}}^\top_*:=({\bf e}^\top_*, -1).$
Hence, the tight lower bound on the dimensions of maximal faces of the ${\mathbb{CP}}(\bar n)$ is less  than  or equal \ to  $\bar n+3.$
 The statement of this theorem follows from this inequality and
Lemma \ref{L2}. $\ \Box$

\vspace{1mm}

\begin{remark}  Note that if we  construct the  matrix $B({\mathcal{A}},I)$  based on  the matrix  ${\mathcal{A}}$ (defined in (\ref{AA}))
using Theorem 3.8 in \cite{Baumert},    that is,  with  any $I\subset [n]$   satisfying  $|I|=1,$
 we would obtain a  coarser  estimate for $low(\bar n)$, namely
$low(\bar n)\leq 2\bar n-3.$  For  even  $\bar n \geq 8,$ this  estimate  improves  upon  that obtained   in  \cite{Zhang},  yet  it  remains weaker
than the estimates    established  above  in this paper.
 This   reveals  one of the advantages of Propositions \ref{P_extreme} and \ref{Pcond}   compared to  Theorem 3.8 in \cite{Baumert}. 

\end{remark}

Recall that in  \cite{Zhang}, the authors show that  the  minimal  dimensions of
 maximal faces of the ${\mathbb{CP}}(\bar n)$,  for $\bar n\geq 6,$ cannot  exceed  $$estim^*(\bar n):=(\bar n^2 -5\bar n+8)/2.$$
It is easy to see that, for even $\bar n\geq 8$,
 our  upper  estimation, $estim(\bar n),$  $$estim(\bar n):=\bar n+3,$$ is  sharper  than  $estim^*(\bar n)$, since $estim^*(\bar n)-estim(\bar n)=
\frac{1}{2}(\bar n^2-7 \bar n+2)>0.$

\vspace{3mm}

Unfortunately, for  an  even  $\bar n\geq 6,$ we   were not able to determine  the exact value  of
 the  tight lower bound on the dimensions  of maximal faces of the cone ${\mathbb{CP}}(\bar n)$.
 However,  we have obtained  fairly sharp   upper estimates for this value.

 In fact,   Theorem \ref{Teven} shows  that  for the tight lower bound on
 the dimensions of maximal faces of the ${\mathbb{CP}}(\bar n)$, the difference between
its upper   and  lower estimates is equal to $3$ for all even $\bar n\geq 6.$

 It follows from Theorems \ref{T5} and \ref{Teven} that
the  tight lower bound on the dimensions of maximal faces of the cone ${\mathbb{CP}}(\bar n)$
  grows  linearly  (and  not quadratically)  with respect to   $\bar n.$

\section{Conclusion}\label{S.5}

In this paper, we investigated maximal faces of the cone of completely positive matrices and established new results concerning their dimensions.
Our approach is based on an explicit construction of maximal faces of the completely positive cone using extremal matrices of the dual copositive cone.

For all odd dimensions $n \ge 5$, we prove that the tight lower bound on the dimensions of  maximal faces of the cone $\cp$ equals $n$.
For even dimensions $n \ge 6$, we show that the tight lower bound satisfies
\[
n \le low(n) \le n+3.
\]
These results provide substantially sharper bounds than those previously available and offer a clearer understanding of the structure of maximal faces of the completely positive cone.

An open problem remains: to determine the exact value, or to obtain a sharper estimate, of the tight lower bound on the dimensions of maximal faces of the cone $\mathbb{CP}(n)$ for even $n \ge 6$.

%\appendix

%\section{Sample Appendix Section}
%Lorem ipsum dolor sit amet, consectetur adipiscing elit, sed do eiusmod tempor section \ref{sec:sample1} incididunt ut labore et dolore magna aliqua. Ut enim ad minim veniam, quis nostrud exercitation ullamco laboris nisi ut aliquip ex ea commodo consequat. Duis aute irure dolor in reprehenderit in voluptate velit esse cillum dolore eu fugiat nulla pariatur. Excepteur sint occaecat cupidatat non proident, sunt in culpa qui officia deserunt mollit anim id est laborum.

 %\bibliographystyle{elsarticle-num}
 %\bibliography{cas-refs}

\end{document}